\documentclass{article}
\usepackage[utf8]{inputenc}
\usepackage{amsmath, amsthm, amssymb}
\usepackage{mathtools}
\usepackage[dvipsnames]{xcolor}
\usepackage{enumerate}

\theoremstyle{plain}
\newtheorem{theorem}{Theorem}
\newtheorem{corollary}[theorem]{Corollary}
\newtheorem{lemma}[theorem]{Lemma}
\newtheorem{proposition}[theorem]{Proposition}
\newtheorem{conjecture}[theorem]{Conjecture}

\theoremstyle{definition}
\newtheorem{definition}[theorem]{Definition}
\newtheorem{example}[theorem]{Example}
\newtheorem{note}[theorem]{Note}
\newtheorem{remark}[theorem]{Remark}

\begin{document}

\title{Cross-Frame Potential}
\author{Roza Aceska\thanks{Department of Mathematical Sciences, Ball State University, raceska@bsu.edu }   \and McKenna Kaczanowski\thanks{Department of Mathematics, Indiana  University,  mckennakaczanowski@gmail.com }  }


\maketitle
\begin{abstract}
 We study  
 a cross-frame potential function, which is closely tied to the cross correlation between a frame and its dual. \color{black}
 We analyze the behavior of this new function to determine what information the cross-frame potential value can reveal about the relationship between the two frames. We also define a cross frame potential function for fusion frames, and study its properties for special types of fusion frames. 
\end{abstract}

\vspace{2.3mm}

{\bf Keywords:} Finite frame, Canonical dual frame, Cross-frame potential 

\vspace{2.3mm}

{\bf AMS subject classification:}  42C15, 42C99

\section{Introduction}


 
Frames \cite{CassFin}
have  properties similar to 
bases, but in general they offer more flexi-bility to accommodate specific
design requirements, \color{black} and their {redundancy}   allows for protection against information loss in data transmission. 
In the past few decades frame representations of vectors have been  studied  by an army of   mathematicians and applied scientists, interested in their mathematical properties and in the advantages they offer in numerous applications \cite{Jelena,  Groechenig}.

{\it Fusion frames} as generalizations of frames were first introduced in \cite{Casazza2004} as  weighted sequences of
subspaces with controlled  {overlaps}. These redundant subspaces 
ease the construction of frames by building them locally 
and then piecing the local frames together by employing a special structure of the set of subspaces.  
The properties and applications of fusion frames (useful in distributed processing) were further explored in \cite{ CasazzaGitta,Calderbank2011, AceskaBouchot} and references within, while  dual fusion frames were  studied in detail in \cite{Heineken}.

The {\it frame potential } function \cite{CasazzaGitta} takes a frame as input and returns its “frame potential value”, which is simply  the $\ell^2$ sum of the inner products between the frame vectors. This value  
relates the geometric structure of the input frame to optimal  properties with respect to frame potential \cite{BodmannHaas}. \color{black}
The concept was first studied in \cite{BenedettoF} where the question of minimizing the frame potential was initially answered - 
 the potential is minimized when the input frame has specific geometry. Further results in this direction can be found  in \cite{BodmannHaas}, where alternative definitions and generalizations of the notion are also pursued. For a study of the frame potential function in finite-dimensional Banach spaces see \cite{ChavezDominguez}. 

   The notion of {\it fusion frame potential} was introduced in \cite{fusionPotential}, where the authors show that if  the fusion frame’s subspaces are large in number but small in dimension compared to the dimension of the underlying
space, then fusion frames will always exist, with each being a minimizer of the fusion frame potential. This topic is further explored in \cite{Massey}; the authors
related this problem to the index of the Hadamard product by positive matrices and use it to give different characterizations of the minimization of the fusion frame potential. In \cite{HeinekenFFP}, the authors provide a description of the local  minimizers which projections are eigenoperators of the fusion frame operator.

 
In this paper we study several ways of evaluating the   cross correlation   between two frames in the finite-dimensional case. The {\it  cross frame potential} is defined  as
\color{black} 
\begin{equation}\label{ourdefCFP}    P_F(G) = \sum_{i=1}^k\sum_{j=1}^k|\langle f_i,g_j \rangle |^2.\end{equation} Observe that \eqref{ourdefCFP} is an adaptation of the classical frame potential function (De- finition~\ref{framepotentialdef}) and it takes the   potential of one frame $F$ with respect to another frame $G$ (one of its duals; Example~\ref{notdualframe1}  points out that  the cross frame potential of a  frame with nonduals will sometimes attain unusual values). This function is minimized for the canonical dual (Section~\ref{crossframepotdefs}). 

 Formula \eqref{ourdefCFP} \color{black}  is naturally related to the {\it  cross-Gramian}  of a frame $F$ with its dual frame $G$, $Gr(F,G) = \left(\langle f_i, g_j \rangle \right)_{i,j=1}^k$, which in general is not a symmetric matrix (Example~\ref{exNonSymmetricGram}).  
Note that the mixed-frame potential function $ \Tilde{FP} (F,G) $ as studied in \cite{HeineMixed}    measures  how close a frame pair is to a biorthogonal system, and it \color{black}   coincides in value with \eqref{ourdefCFP} in a special case:  
In the real domain 
for frame pairs with symmetric cross Grammians (such as the frame pairs in Example~\ref{GrassmannEx} and Example~\ref{exNonSymmetricGram}),
%
%
\begin{center}
{\it  $P_F (G) = \tilde{FP} (F,G)$ when $Gr(F,G)^T= Gr(F,G)$}, \end{center}
%
%
which holds  for the canonical dual (as  in \color{black} Proposition~\ref{generalresult}).   Thus,  the smaller the value of \eqref{ourdefCFP} is, the closer the frame pair $(F, G)$ is to a biorthogonal system. \color{black}

\color{black}



Motivated by the various ways of assessing frame potential in \cite{BodmannHaas} and what its value says about the quality of the frame at hand, we extend the ideas of {\it $p-$th frame potential}  and {\it exponential potential} to frame pairs in Subsection~\ref{crossPpotential} and Section~\ref{exponPot}, where we highlight the special values for the cross frame potential with the canonical dual. In Section~\ref{exponPot} we also point to some unanswered questions regarding the off-diagonal maximal magnitude and state a couple of conjectures, which can be seen as generalizations of results on the Gramian \cite{Welch, Kasso}.

We   introduce the notion of cross potential for fusion frames in Section~\ref{CFFPdef},  where we  generalize the idea of fusion frame potential 
as discussed in 
\cite{fusionPotential}, and we obtain exact computations of the cross potential for special types of fusion frame pairs.

\section{Background } 
By $\mathbb{H}=\mathbb{F}^n$  we denote a  finite-dimensional Hilbert space ($\mathbb{R}^n$ or $\mathbb{C}^n$). 

\begin{definition} \label{framedefinition}
A sequence of vectors $F=\{f_i\}_{i=1}^k$, with $k \geq n$, in an $n$-dimensional Hilbert space $\mathbb{H}$ is a \textit{frame for $\mathbb{H}$} if   there exist  real constants $0 < A \leq B < +\infty$ such that  
\begin{equation} A\lVert f \rVert ^2 \leq \sum_{i=1}^k|\langle f,f_i\rangle|^2 \leq B \lVert f \rVert ^2\end{equation} for every $f \in \mathbb{H}$. 
$A$ is called the lower frame bound for $\{f_i\}_{i=1}^k$, and $B$ is called the upper frame bound.
\end{definition}
In a  finite-dimensional space $\mathbb{H}$, frames are simply  spanning sets of $\mathbb{H}$.
All bases are frames by definition, but frames in general are not bases. 
 
\begin{example} \label{frameExample1}
$F = \Bigg \{\begin{bmatrix}
0 \\
1  
\end{bmatrix},
\begin{bmatrix}
1 \\
1 
\end{bmatrix},
\begin{bmatrix}
-1 \\
1
\end{bmatrix} \Bigg \}$ is a frame for $\mathbb{R}^2$, but 
not a basis. 
\end{example} 
 
%
%

Each frame has an associated analysis, synthesis, and frame operator.
Given a frame $F = \{f_i\}_{i=1}^k$ for $\mathbb{F}^n$,  
the   {\it analysis operator} \color{black} $\theta_F$ 
is defined for every $f \in \mathbb{H}$ by 
$( \theta_F f)_i = \langle f,f_i \rangle$;   
the matrix of $\theta_F$ has the frame elements of $F$ as its rows. 
 The \textit{synthesis operator}   is the adjoint 
 operator of 
 $\theta_F$,  indicated by $\theta^*_F$ \color{black};  its matrix has the  elements of frame $F$ as columns, and we will label this matrix by $F$ for simplicity.  
%
 %
The \textit{frame operator} of  $F = \{f_i\}_{i=1}^k$ is 
defined as $ S = \theta^*_F\theta_F$. 
It shows that the frame operator is a positive, self-adjoint, invertible operator. 



In signal processing, a frame is  used to compute the analysis coefficients, while a {\it dual frame} is used for synthesis, to reconstruct the signal (or the other way around). 
\begin{definition}
Let $\{f_i\}_{i=1}^k$ be a frame for $\mathbb{H}$. A \textit{dual frame} for $\{f_i\}_{i=1}^k$ is a frame $\{g_i\}_{i=1}^k$ such that for every $f \in \mathbb{H}$, \begin{equation*}
    f = \sum_{i=1}^k\langle f,g_i \rangle f_i = \sum_{i=1}^k\langle f,f_i \rangle g_i.
\end{equation*}
\end{definition}
%
Thus, 
a frame $G = \{g_i\}_{i=1}^k$  is a dual frame for $F$ if $\theta_F^*\theta_G = \theta_G^*\theta_F$ is the identity map on $\mathbb{H}$; that is, $FG^T=GF^T=I_{n\times n}$.
 

A frame can have many dual frames. 
The \textit{canonical dual frame} for a frame $F = \{f_i\}_{i=1}^k$ for $\mathbb{H}$ with frame operator $S$ is the frame $\tilde{F} = \{S^{-1}f_i\}_{i=1}^k$. 
The frame operator for $\tilde{F}$ is $S^{-1}$.
The canonical dual frame has a minimization property:

\begin{lemma}\label{minimizationlemma} 
Let $F = \{f_i\}_{i=1}^k$ be a frame for $\mathbb{H}$, let $\tilde{F} = \{\tilde{f}_i\}_{i=1}^k$ be the canonical dual frame for $F$, and let $H = \{h_i\}_{i=1}^k$ be a dual frame (not necessarily the canonical dual) for $F$. Then for every $f \in \mathbb{H}$, $$\sum_{i=1}^k|\langle f,\tilde{f}_i \rangle|^2 \leq \sum_{i=1}^k|\langle f,h_i \rangle|^2,$$ with equality if and only if $H = \tilde{F}$.
\end{lemma}

The {\it cross-Gramian } $Gr(F,G)=\theta_F\theta_G^*$ is a bounded, idempotent 
operator on $\mathbb{F}^n$ whenever $G$ is a dual frame of $F$; it is self-adjoint if $G$ is the canonical dual frame of $F$.   The matrix of $Gr(F,G)$ is $F^TG$; thus  $Gr(F,G)^*= Gr(G,F)$.
 
A frame    is called \textit{tight} if   $A=B$ in Definition \ref{framedefinition}. One useful feature of tight frames is that its frame operator $S = \theta^*\theta$ 
is a scalar multiple of the identity operator, where that scalar is the frame bound $A=B$. Thus, the canonical dual frame   of a tight frame $F = \{f_i\}_{i=1}^k$,  is simply 
$\{\frac{1}{A}f_i\}_{i=1}^k$. When a frame $F$ is tight, the representation of a vector $f$ in $\mathbb{H}$ depends only on $F$, 
that is, we only need $F$ for the analysis and synthesis. 
\begin{example} \label{tightFrameExample}
$$F = \Bigg\{ \begin{bmatrix}
0 \\
1
\end{bmatrix},
\begin{bmatrix}
-\sqrt{3}/2 \\
-1/2
\end{bmatrix},
\begin{bmatrix}
\sqrt{3}/2 \\
-1/2
\end{bmatrix} \Bigg\}$$
is a tight frame for $\mathbb{R}^2$, 
with frame operator   $S=\frac{3}{2}I_2$; its canonical dual is $\frac{2}{3}F$.
\end{example}


 
\subsection{Frame potential} 
The frame potential of a finite sequence of vectors is a theoretical version of potential energy in physics developed in \cite{BenedettoF} and later studied in \cite{BodmannHaas, ChavezDominguez, Massey}. 
In physics, a system of objects acting under a force will move to minimize its potential energy.   
\begin{definition} \label{framepotentialdef}
The \textit{frame potential} of a frame $F = \{f_i\}_{i=1}^k$ for $\mathbb{H}$, $P_F$, is $$P_F = \sum_{i=1}^k\sum_{j=1}^k|\langle f_i,f_j \rangle|^2.$$
\end{definition}
The frame potential function   can also be written as the trace of the square of the frame operator: $P_F = tr(S^2).$

\begin{proposition}  
\label{framepotentialvalue}
Let $F = \{f_i\}_{i=1}^k$ be a sequence of vectors in an n-dimensional Hilbert space $\mathbb{H}$ with $\sum_{i=1}^k\lVert f_i \rVert ^2 = L$. Then we have
\begin{equation*} P_F \geq \frac{L^2}{n},\end{equation*} with equality if and only if $F$ is a tight frame. 
\end{proposition} 
Note that the tight frame $F$ from  Example \ref{tightFrameExample}  
has frame potential $P_F = 9/2$ which is the expected minimum value  (as $L = 1 + 1 + 1 = 3$). 
The  frame $F$ from Example \ref{frameExample1}  has frame potential $P_F = 13$, which is greater than the minimum value $25/2$ (as expected for non-tight frames).

\begin{proposition}\cite{BodmannHaas}\label{BHdiag}
Let $G=Gr(F,F)$ be the Grammian for a Parseval frame $F$ for $\mathbb{F}^n$, 
 consisting of $k$ vectors. Then 
$$\sum_{i=1}^k |G_{i,i}|^2 \geq \frac{n^2}{k},$$
with equality if and only if $G_{i,i} = \frac{n}{k}$ for each $1 \leq i \leq k$.
\end{proposition}
\vspace{2.13mm}

Other ways to compute  frame potential were explored in \cite{BodmannHaas}:
\begin{definition} 
The $p$-th frame potential of a frame $F = \{f_j\}_{j=1}^k$ for an $n$-dimensional 
Hilbert space $\mathbb{H}$ is given by 
$$\phi_p(F) = \sum_{i,j = 1}^k |\langle f_i,f_j \rangle |^{2p} \; \text{ for some } \; p>0.$$
\end{definition}
\begin{proposition}\cite{BodmannHaas}
If two frames $F$ and $F'$ are unitarily equivalent, then the $p$-th frame potential of $F$ is equal to the $p$-th frame potential of $F'$.
\end{proposition}

\begin{theorem}\label{HaasPotent}
\cite{BodmannHaas} Let $F = \{f_i\}_{i=1}^k$ be a frame for $\mathbb{F}^n$,  
and $\lVert f_i \rVert ^2 = n/k$ for all $i \in \mathbb{Z}_n$, and let $p > 1$. Then 
$$\phi_p (F) = \sum_{i,j=1}^k |\langle f_i,f_j \rangle |^{2p} \geq \frac{n^{2p}(k-1)^{p-1} + n^p(k-n)^p}{(k-1)^{p-1}k^{2p-1}}$$
and equality holds if and only if $F$ is an equiangular Parseval frame.
\end{theorem}
\subsection{Fusion frames and 
{fusion frame potential} } 
Fusion frames are a generalization of traditional frames that are well-suited for distributed data processing: breaking up the processing job into chunks that different computers execute at the same time and then combining the results; 
the reason for this is that fusion frames are sequences of subspaces instead of vectors. Because of this, we can think of fusion frames as frames made up of frames for the subspaces.    
\begin{definition} \label{fusion_def} 
A sequence of closed subspaces $\{W_i\}_{i=1}^k$ of an $n$-dimensional Hilbert space $\mathbb{H}$ is a \textit{fusion frame for $\mathbb{H}$} if there exist constants $0 < A \leq B < \infty$ such that for every $f \in \mathbb{H}$ it holds
$$A\lVert f \rVert ^2 \leq \sum_{i =1}^k{\lVert P_i f \rVert ^2} \leq B\lVert f \rVert ^2, $$   where $P_i$ is the orthogonal projection onto subspace $W_i$ for every $i \in I=\{1,2, \hdots, k\}$. 
A fusion frame is called \textit{tight} if $A = B$. 
 
\end{definition}

We can  view a fusion frame $\{W_i\}_{i=1}^k$ as  a collection $\{P_i\}_{i=1}^k$ of orthogonal projections onto the respective subspaces.
The fusion frame operator $S$ of a frame $\{P_i\}_{i=1}^k$ is the sum of the projections: $$S = \sum_{i=1}^k{P_i}.$$
%
Given a finite sequence of subspaces $\{W_i\}_{i \in I}$ of $\mathbb{H}$, if $\{f_i^j\}_{j \in J_i}$   is a frame for subspace $W_i$ for an index set $J_i$, $i \in I$, then the union of the vectors in every $\{f_i^j\}_{j \in J_i}$ is a  frame for $\mathbb{H}$ if and only if $\{W_i\}_{i=1}^k$ is a fusion frame for $\mathbb{H}$ \cite{CasazzaGitta}.

\begin{example}
Let $W_1=\{(x,y,0)^T | x, y \in \mathbb{R}\}$, $W_2=\{(x,y,y)^T | x, y \in \mathbb{R}\}$. Then $P = \{W_1,W_2\}$ is a fusion frame for $\mathbb{R}^3$. We can also view $P$ as the sequence $\{P_1,P_2\}$ of orthogonal projections onto $W_1$ and $W_2$, where 
$$P_1 = \begin{bmatrix}
1 & 0 & 0 \\
0 & 1 & 0 \\
0 & 0 & 0 
\end{bmatrix}, P_2 = \begin{bmatrix}
1 & 0 & 0 \\
0 & \frac{1}{2} & \frac{1}{2} \\
0 & \frac{1}{2} & \frac{1}{2} 
\end{bmatrix}, \; \text{ while} \; S = P_1 + P_2 = \begin{bmatrix}
2 & 0 & 0 \\
0 & \frac{3}{2} & \frac{1}{2} \\
0 & \frac{1}{2} & \frac{1}{2}
\end{bmatrix}$$
is the fusion frame operator for $P$. 
We can take any frame $F_1$ for $W_1$ and any frame $F_2$ for $W_2$, and the union $F = F_1\cup F_2$   will be a traditional frame for $\mathbb{R}^3$. For example, let 
$$F_1 = \Bigg\{ \begin{bmatrix}
1 \\
0 \\
0 
\end{bmatrix},
\begin{bmatrix}
0 \\
1 \\
0
\end{bmatrix}, \begin{bmatrix}
2 \\
-1 \\
0
\end{bmatrix} \Bigg\}, F_2 = \Bigg\{ \begin{bmatrix}
0 \\
1 \\
1 
\end{bmatrix},
\begin{bmatrix}
1 \\
0 \\
0
\end{bmatrix} \Bigg\}.
$$
Then $F_1 \cup F_2$ is a traditional frame for $\mathbb{R}^3$.
\end{example}
 
\begin{definition}\cite{fusionPotential}
The fusion frame potential  of a fusion frame $P = \{P_i\}_{i=1}^k$, where $P_i$ are orthogonal projections, is $$FFP(P) = Tr(\sum_{i=1}^k{P_i})^2.$$ 
Equivalently, $$FFP(P) = \sum_{i=1}^k\sum_{j=1}^k Tr(P_iP_j).$$
\end{definition}
Similarly to traditional frame potential, fusion frame potential is minimized exactly when the fusion frame is tight. 
\begin{proposition}\label{fusionframepotentialvalue}\cite{fusionPotential}
Let $P = \{P_i\}_{i=1}^k$ be a sequence of orthogonal projections $P_i: \mathbb{F}^n \rightarrow \mathbb{F}^n$ with $Tr(P_i) = L_i$, where $\mathbb{F} = \mathbb{R}$ or $\mathbb{C}$. Then $$FFP(P) \geq \frac{1}{n}(\sum_{i=1}^k L_i)^2,$$ with equality if and only if $\{P_i\}_{i=1}^k$ is a tight fusion frame. 
\end{proposition}
In Proposition \ref{fusionframepotentialvalue}, each $L_i$ is also the rank of $P_i$ since the trace of an orthogonal projection is equal to its rank - which means: $L_i = \dim W_i$,  $1\leq i\leq k$. 

\subsubsection{On computing the canonical dual of a fusion frame}

Viewing a fusion frame $P$ as a sequence of subspaces $\{W_i\}_{i=1}^k$, the canonical dual frame $Q$ is $\{S^{-1}W_i\}_{i=1}^k$, where $S$ is the fusion frame operator for $P$. This  makes sense, since it is the same process we followed to compute the canonical dual of a traditional frame. However, if the sequence of orthogonal projections onto each subspace is $\{P_i\}_{i=1}^k$, then $\{S^{-1}P_i\}_{i=1}^k$  will not necessarily  give us a fusion frame, since 
the products $S^{-1}P_i$ are not necessary orthogonal projections. 

In practice, when we want to  state the canonical dual fusion frame as a sequence of orthogonal projections $\{Q_i\}_{i=1}^k$, 
we first compute the subspaces $S^{-1}W_i$ of $Q$ and then compute the orthogonal projection operators $Q_i$ for   subspaces $S^{-1}W_i$. 

\section{Cross frame potential}\label{crossframepotdefs}
We build on the notion of frame potential function (Definition~\ref{framepotentialdef}) to study the potential of one frame \textit{with respect to} another frame. We call this new function the {\it cross frame potential} function. 
\begin{definition}Let $F = \{f_i\}_{i=1}^k$ and $G = \{g_i\}_{i=1}^k$ be frames for $\mathbb{F}^n$. The \textit{cross frame potential} $P_F(G)$ of $G$ with respect to $F$ is 
\begin{equation} P_F(G) = \sum_{i=1}^k\sum_{j=1}^k|\langle f_i,g_j \rangle |^2.\end{equation}
\end{definition}
By definition, $P_F(G)=P_G(F)$. Clearly, the cross-frame potential of a frame $F$ with itself is just the frame potential of $F$ as per Definition~\ref{framepotentialdef}. 
For a Parseval frame $F$, Proposition~\ref{framepotentialvalue} implies that $P_F(F) = L^2/n$, which is the minimal value of the frame potential function.  
As the canonical dual frame of a Parseval frame $F$ is $F$, 
and its frame potential is the minimal value, then it makes sense to expect that the cross frame potential function $P_F(G)$ for any frame $F$ is minimized when the input frame $G$ is its canonical dual: 
  
\begin{proposition} \label{generalresult}
Let $F$ be a frame for $\mathbb{F}^n$ and let $\tilde{F}$ be the canonical dual frame for $F$. Then the cross- frame potential of $\tilde{F}$ with respect to $F$ is equal to $n$. That is, $$P_F(\tilde{F}) = n.$$
\end{proposition}

\begin{proof}
Let $G = Gr(F,\tilde{F})=\theta_F\theta_{\tilde{F}}^*$ denote the cross-Grammian of $F$ with $G$. 
 Observe that $P_F(\tilde{F})$ is equal to the sum of the squared entries of $G$, and so 
$$P_F(\tilde{F}) = Tr(G^*G).$$
Since $\tilde{F}$ is the canonical dual for $F$, $G$ is self-adjoint, since $\theta_{\tilde{F}}^* = S^{-1}\theta_F^*$, where $S$ is the frame operator for $F$, and so 
$$G^* = (\theta_F\theta_{\tilde{F}}^*)^* = \theta_{\tilde{F}}\theta_F^* = (S^{-1}\theta_F^*)^*\theta_F^* = \theta_F(S^{-1})^*\theta_F^* = \theta_FS^{-1}\theta_F^* = \theta_F\theta_{\tilde{F}}^* = G,$$ 
as $S$ is an invertible, self-adjoint operator. In addition, because $F$ and $\tilde{F}$ are dual to one another, $\theta_{\tilde{F}}^*\theta_F = I_n$, and so 
\begin{equation}\label{projectionS} 
G^2 = (\theta_F\theta_{\tilde{F}}^*)(\theta_F\theta_{\tilde{F}}^*) = \theta_F(\theta_{\tilde{F}}^*\theta_F)\theta_{\tilde{F}}^* = \theta_FI_n\theta_{\tilde{F}}^* = G.\end{equation}
Thus $G$ is an idempotent matrix, and so its trace is equal to its rank. Therefore,
\begin{equation} \label{rankS}
P_F(\tilde{F}) = Tr(G^*G) = Tr(G) = rank(G).\end{equation}
Since $\theta_{\tilde{F}}^*$ is an $n \times k$ matrix of rank $n$ (the columns of $\theta_{\tilde{F}}^*$ form a frame for $\mathbb{F}^n$), $rank(\theta_F\theta_{\tilde{F}}^*) = rank(\theta_F) = n$, as the rows of $\theta_F$ form a frame for $\mathbb{F}^n$. Thus $P_F(\tilde{F}) = n$ when $\tilde{F}$ is the canonical dual frame for $F$.
\end{proof}

The cross-frame potential   $P_F(H)$ is minimized exactly when $H$ is the canonical dual frame of $F$;   an equivalent result can be found in \cite{ChrisDatta}:  \color{black}

\begin{theorem} \label{mainresult} 
Let $F$ be a frame for $\mathbb{F}^n$, and let $H$ be a dual frame for $F$. Then \begin{equation}P_F(H) \geq n, \end{equation} with $P_F(H) = n$ if and only if $H$ is the canonical dual frame for $F$. 
\end{theorem}
 
The proof of Thoerem~\ref{mainresult} follows from  Lemma \ref{minimizationlemma} and  Proposition \ref{generalresult}.
 \color{black}
 
\begin{example}\label{example2revisit} 
The canonical dual frame $G$ of the frame $F$ in  Example \ref{frameExample1} 
is
$$G = \Bigg\{ \begin{bmatrix}
0 \\
\frac{1}{3}
\end{bmatrix},
\begin{bmatrix}
\frac{1}{2} \\
\frac{1}{3} 
\end{bmatrix},
\begin{bmatrix}
-\frac{1}{2} \\
\frac{1}{3}
\end{bmatrix} \Bigg\}, \; \text{while } \; H = \Bigg\{ \begin{bmatrix}
0 \\
1
\end{bmatrix}, \begin{bmatrix}
\frac{1}{2} \\
0
\end{bmatrix}, \begin{bmatrix}
-\frac{1}{2} \\
0
\end{bmatrix} \Bigg\}$$ is another dual frame of $F$. 
Then $P_F(G) = 2$, as expected, and  $P_F(H) = 4 > 2$.
\end{example}

The minimization result of Theorem~\ref{mainresult} only holds when the domain of the cross potential function $P_F$ is limited to dual frames for $F$. We give two examples of cross frame potential computations with non-dual frames that match or fall below the minimum value $n$:
\begin{example}\label{notdualframe1}
Let $F = \{f_i\}_{i=1}^k$ be a frame for $\mathbb{R}^n$, and let $G = \{g_i\}_{i=1}^k$ be its canonical dual. Suppose $\langle f_i,g_1 \rangle \neq 0$ for some $f_i$ in $F$. Then  the frame $H = \{\frac{g_1}{2},g_2,...,g_k\}$ is not a dual frame for $F$. In particular, we have 
\begin{align*} P_F(H) &= \sum_{i=1}^k\left(|\langle f_i,\frac{1}{2}g_1 \rangle|^2 + \sum_{j=2}^k|\langle f_i,g_j \rangle|^2\right) \\&= \frac{1}{4}\sum_{i=1}^k|\langle f_i,g_1 \rangle|^2 + \sum_{i=1}^k\sum_{j=2}^k|\langle f_i,g_j \rangle|^2 \\
&< \sum_{i=1}^k|\langle f_i,g_1 \rangle|^2 + \sum_{i=1}^k\sum_{j=2}^k|\langle f_i,g_j \rangle|^2  = \sum_{i=1}^k\sum_{j=1}^k|\langle f_i,g_j\rangle|^2 = P_F(G) = n.
\end{align*}
So $P_F(H)<n$; $H$ is not a dual frame for $F$  by Theorem   \ref{mainresult}. 
\end{example}
\begin{example}\label{notdualframe2}
Let $F = \{f_i\}_{i=1}^k$ be a frame for $\mathbb{R}^n$, and let $G = \{g_i\}_{i=1}^k$ be its canonical dual, with $g_1 \neq 0$. Then    $H = \{-g_1,g_2,...,g_k\}$ is not a dual frame of $F$. Yet, we have 
\begin{align*} P_F(H) &= \sum_{i=1}^k\left(|\langle f_i,-g_1 \rangle|^2 + \sum_{j=2}^k|\langle f_i,g_j \rangle|^2\right)\\ 
&= \sum_{i=1}^k\left(|\langle f_i,g_1 \rangle|^2 + \sum_{j=2}^k|\langle f_i,g_j \rangle|^2\right)  = \sum_{i=1}^k\sum_{j=1}^k|\langle f_i,g_j \rangle|^2  = P_F(G)  = n.
\end{align*}
Since $P_F(H)=n$, we conclude that $H$ cannot be an alternate dual for $F$;  if it was, we'd have $P_F(H) >n$. This example illustrates how when the domain of $P_F(G)$ is extended to include non-duals, the cross potential value $n$ is no longer unique. From this example we can also see that, for any frame $F$, reversing the sign of any number of nonzero vectors of the canonical dual will never result in a dual frame for $F$. 
\end{example}

  \begin{definition}
Let $\mathbb{H}$ be a  Hilbert space. Two frames for $\mathbb{H}$, $F = \{f_j\}_{j \in J}$ and $F' = \{f'_j\}_{j \in J}$,   are \textit{unitarily equivalent} if there exists an orthogonal ($\mathbb{H}$ is a real space) or unitary ($\mathbb{H}$ is complex) operator $U$ on $\mathbb{H}$ such that $f_j = Uf'_j$ for all $j \in J$.
\end{definition}

\begin{proposition}
Let $F,G$ be frames for $\mathbb{F}^n$. Let  $U:\mathbb{F}^n \rightarrow \mathbb{F}^n$ be a unitary operator. 
Then $UF$ and $UG$ are also frames for $\mathbb{F}^n$, and $$P_F(G) = P_{UF}(UG).$$
\end{proposition}

\begin{proof}
Let $F = \{f_i\}_{i=1}^k$ and $G = \{g_i\}_{i=1}^k$ be frames for $\mathbb{F}^n$. If $U$ is an unitary operator, then $UF$ and $UG$ are frames for $\mathbb{F}^n$ as well. %
For all $i,j \in \{1,...,k\}$, we have
$$\langle Uf_i,Ug_j \rangle = (Uf_i)^*Ug_j = f_i^*U^*Ug_j = f_i^*(U^*Ug_j) = \langle f_i,U^*Ug_j \rangle = \langle f_i,g_j \rangle.$$ These are exactly the inner products in the sum for $P_{UF}(UG)$, so the result follows. 
\end{proof}

The quality of the diagonal of the cross-Gramian preserves the analogous  result    from  Proposition~\ref{BHdiag} on  Gramians: 
\color{black}
\begin{proposition}\label{prop29}
Let $F = \{f_i\}_{i=1}^k$ be a frame for $\mathbb{F}^n$, and let $G = \{g_i\}_{i=1}^k$ be a dual frame for $F$. Then 
$$\sum_{i=1}^k | \langle f_i,g_i \rangle |^2 \geq \frac{n^2}{k},$$
with equality if and only if $\langle f_i,g_i \rangle = \frac{n}{k}$ for every $1 \leq i \leq k$.
\end{proposition}

\begin{proof}
Applying Jensen's inequality, we have
$$\sum_{i=1}^k|\langle f_i,g_i \rangle |^2 \geq \frac{(\sum_{i=1}^k\langle f_i,g_i \rangle)^2}{k},$$
with equality if and only if there is some constant $C > 0$ such that $\langle f_i,g_i \rangle = C$ for every $1 \leq i \leq k$. 
As stated in the proof of Proposition \ref{generalresult}, when $G$ is dual to $F$, $(\theta_F\theta_G^*)^2 = \theta_F\theta_G^*$, and so the trace of $\theta_F\theta_G^*$ is equal to its rank, which is $n$. Therefore we have
$$\sum_{i=1}^k|\langle f_i,g_i \rangle |^2 \geq \frac{n^2}{k},$$
with equality if and only if $\langle f_i,g_i \rangle = \frac{n}{k}$ for every $1 \leq i \leq k$. 
\end{proof}

A frame pair $(F, \tilde{F}) $ which satisfies the equality in Proposition~\ref{prop29} is studied later in Example~\ref{GrassmannEx}.

\subsection{Cross $p$-th frame potential}\label{crossPpotential}

\begin{definition}
The \textit{cross $p$-th frame potential of a frame $F$ with respect to a dual frame $G$} for $\mathbb{F}^n$ is given by 
$$\phi _p (F,G) = \sum_{i,j=1}^k | \langle f_i,g_j \rangle |^{2p}.$$
\end{definition}

\begin{theorem}\label{whatsL}
Let $F = \{f_i\}_{i=1}^k$ be a frame for $\mathbb{F}^n$, and let $G = \{g_i\}_{i=1}^k$ be a dual frame for $F$     
such that $\langle f_i,g_i \rangle = C_1$ for all $1 \leq i \leq k$, where $C_1 > 0$ is a constant, and let $p \geq 1$. Then
\begin{equation}\label{equalThm39}
\phi _p (F,G) \geq \frac{(nk-n^2)^p + n^{2p}(k-1)^{p-1}}{k^{2p-1} (k-1)^{p-1}}, \end{equation}
with equality in \eqref{equalThm39} if and only if 
\begin{itemize} 
\item[i.] $G$ is the canonical dual frame of $F$,  
and  

\item[ii.] there exists some $C_2 > 0$ such that $| \langle f_i,g_j \rangle |^2 = C_2$ for all $i,j \in \{1,...k\}$ with $i \neq j$.
\end{itemize}
 
\end{theorem}

Note that in  Theorem~\ref{HaasPotent}, the $p-$th frame potential of a frame $F$ is minimized when $F$ is an equiangular Parceval frame. In Theorem~\ref{whatsL} we cannot tie the minimization to the angles between the frame vectors because, even if   the \color{black} frame elements of $F$ have equal norm, the dual frame elements are not necessarily of equal norm.

\color{black}

\color{black}
\begin{proof}

By \eqref{projectionS}, $Tr(\theta_F\theta_G^*)$ is equal to its rank, which is $n$, as explained in the proof of Proposition \ref{generalresult}. 
Then since $\sum_{i=1}^k \langle f_i,g_i \rangle = Tr(\theta_F\theta_G^*)$, the only way to have $\langle f_i,g_i \rangle = C_1$ for all $1 \leq i \leq k$ is for $C_1$ to be $\frac{n}{k}$. 

Now consider $\sum_{i \neq j} |\langle f_i,g_j \rangle |^{2p}$.
Applying  Jensen's Inequality, we have 
$$\sum_{i \neq j} |\langle f_i,g_j \rangle |^{2p} \geq \frac{(\sum_{i \neq j} |\langle f_i,g_j \rangle |^2)^p}{k^{p-1} (k-1)^{p-1}} = \frac{1}{k^{p-1} (k-1)^{p-1}}\left(\phi _1 (F,G) - \frac{n^2}{k}\right)^p,$$
with equality if and only if there exists $C_2 > 0$ such that $|\langle f_i,g_j \rangle |^2 = C_2$ for all $i,j \in \{1,...,k\}$ with $i \neq j$. Adding this to $\displaystyle \sum_{i=1 }^k |\langle f_i,g_i \rangle |^{2p}$ and applying Theorem \ref{mainresult}, the result follows.

%
\end{proof}

As stated in the proof of Theorem~\ref{whatsL}, the only way for two duals $F,G$ to satisfy the condition  $\langle f_i,g_i \rangle = C_1$ for all $i \in \{1,...k\}$ is when  $C_1 =\frac{n}{k}$. Observe  that in the case of equality in \eqref{equalThm39}, we must have 
\begin{equation} \label{C2value}
C_2=|\langle f_i,g_j \rangle |^2 = \frac{nk-n^2}{k^2(k-1)} \; \text{ for all}  \; 1 \leq i \neq j \leq k.\end{equation}

\begin{example}\label{MerzedesISnice}
The { Mercedes} \color{black} frame $F$  (Example~\ref{tightFrameExample}) and its canonical dual $\frac{2}{3}F$ satisfy  items $i.$ and $ii.$ in  Theorem~\ref{whatsL}  as $\langle f_i,g_i \rangle =  2/3$,   $\langle f_i,g_j \rangle =\pm 1/3$ for $i \neq j$. This was fairly expected, since $F$ is an equiangular tight frame (as required in Theorem~\ref{HaasPotent}).  
\end{example}

\begin{example}\label{niceEx}

It shows that any frame $F=\{f_i\}_{i=1}^3$ of  $\mathbb{R}^2$ and its dual frame $G=\{g_i\}_{i=1}^3$ would satisfy items $i.$ and $ii.$ in  Theorem~\ref{whatsL} 
 if $\langle f_i,g_i \rangle =  2/3$   and   $\langle f_i,g_j \rangle =\pm 1/3$ for $i \neq j$.
  For instance, given a frame  $F$:

$$ \; F= \begin{bmatrix}
 1 & 0 & 1\\
 0 & 1 & 1
\end{bmatrix}, \; \text{its canonical dual is } \; \tilde{F}=\frac{1}{3}\begin{bmatrix}
2 & -1 & 1\\
-1 & 2 & 1
\end{bmatrix},$$
and items $i.$ and $ii.$ are satisfied even though $F$ is not an equiangular nor an equal-norm frame. 
Note that if  we work with a  unit-norm \color{black} frame $F_0$ here (normalized $F$ from above), the canonical dual $G_0$ is not an equal norm frame: 
 \begin{equation*}
     F_0=\begin{bmatrix}
     1 & 0 & 2^{-1/2}\\
     0 & 1 & 2^{-1/2}
     \end{bmatrix},  \, G_0=\begin{bmatrix}
     3 & -1 & 2^{1/2}\\
     -1 & 3 & 2^{1/2}
     \end{bmatrix}, 
 \end{equation*}
 and their cross-Gramian does not satisfy conditions $i.$ and $ii.$  
\end{example} 
\begin{remark}

   The frame $F$ in Example~\ref{niceEx} can be interpreted  as a weighted frame generated from $F_0$ (with weights $1, 1, 2^{-1/2}$), so it is worth exploring the answer to the following question: 
   Which types of frames (after weighing) will satisfy items $i.$ and $ii.$ of Theorem~\ref{whatsL}?

   
   \end{remark} 
   
\begin{remark} 
Any equi-norm tight frame and its canonical dual will satisfy the condition $\langle f_i,g_i \rangle = \frac{n}{k}$ for all $i$; also,  any equiangular tight frame and its canonical dual will satisfy the condition $|\langle f_i,g_j \rangle|^2 = C_2$ for all $i \neq j$. If $F = \{f_i\}_{i=1}^k$ is an equi-norm tight frame with frame bound $A$, then its canonical dual will be $G = \{\frac{1}{A}f_i\}_{i=1}^k$, and so for all $i$, we have 
$$\langle f_i,g_i \rangle = \langle f_i,\frac{1}{A}f_i \rangle = \frac{1}{A}\langle f_i,f_i \rangle,$$
and since $F$ is equi-norm, this value will be the same for all $i$.

Similarly, if $F$ is an equiangular tight frame, then for all $i \neq j$, we have
$$|\langle f_i,g_j \rangle| = |\langle f_i,\frac{1}{A}f_j\rangle| = \frac{1}{A}|\langle f_i,f_j \rangle |,$$ and this value is the same for all $i \neq j$ since $F$ is equiangular. 
\end{remark}

 \section{ Off-Diagonal Magnitude}\label{exponPot}
 
Let $F=\{f_1,\hdots, f_k\}$ be a frame for $\mathbb{F}^n$, and let $H=\{h_1,\hdots, h_k\}$ be a dual frame for $F$. We denote the cross-Gramian of $F$ and $H$ by $Gr(F,H)$ and we denote the {\it maximal off-diagonal magnitude } of $Gr(F,H)$ as 
\[\mu(Gr(F,H)):=\max_{i\neq j}\vert\langle f_i,h_j\rangle \vert.\]

\color{black}

\begin{example}\label{exNonSymmetricGram}
 
The cross Gramians of the frame $F$  from Example~\ref{frameExample1} with its canonical dual $G$ and the other dual frame $H$ as listed in Example~\ref{example2revisit} are
$$ Gr(F,G)=\frac{1}{6}\begin{bmatrix}
2 & 2 & 2\\
2 & 5 & -1\\
2 & -1 & 5
\end{bmatrix}, \;\;\; Gr(F, H) = \frac{1}{2}\begin{bmatrix}
2 & 0 & 0\\
2& 1 & -1\\
2 & -1 & 1
\end{bmatrix},$$  with 
$\mu(Gr(F,G))=\frac{1}{3}$, 
while  $\mu(Gr(F,H))=1$.  In fact, any dual frame $D$ of $F$ has frame elements: 
$$ \begin{bmatrix}
x\\y
\end{bmatrix}, \; \frac{1}{2}\begin{bmatrix}
 {1-x} \\  {1-y} 
\end{bmatrix}, \frac{1}{2}\begin{bmatrix}
 {-1-x} \\   {1-y} 
\end{bmatrix}, \; \text{for some real $x$, $y$};$$
%
the cross Gramian of $F$ and any dual frame $D$ is  
$$ Gr(F,D)=\frac{1}{2} \begin{bmatrix}
2y & 1-y & 1-y \\
2(x+y) & 2- ( x+y) & \ -x-y  \\
2(y-x) & x-y & x-y+2
\end{bmatrix}.$$
If for some $x$, $y$ we have $\mu(Gr(F,D)) < \frac{1}{3}$,   we'd require 
\begin{equation}\label{maxvalD}
 \max \{  \vert x+y \vert, \vert y-x\vert,  \frac{\vert 1-y\vert }{2}  \} <\frac{1}{3}.  
\end{equation}
If both $-\frac{1}{3}< x+y<\frac{1}{3}$ and $-\frac{1}{3}<y-x<\frac{1}{3}$ then $-\frac{2}{3}<2y<\frac{2}{3}$, that is $\vert y\vert <\frac{1}{3}$.  But also, $-\frac{1}{3}<\frac{1-y}{2}<\frac{1}{3}$, which implies $-\frac{2}{3}<1-y<\frac{2}{3}$, that is $\frac{5}{3}>y>\frac{1}{3}$, which is impossible. Thus  
$$\min \{\mu(Gr(F,D)) \, | \, \text{ $D$ is a dual frame for $F$} \}= \frac{1}{3},$$ which we see happens with the canonical dual $D=G$. However, this minimization is not  exclusive; the columns of  
\begin{equation}
    \theta_d^* =\begin{bmatrix}
    0 & \frac{1}{2} & -\frac{1}{2}\\
    \frac{1}{3} & \frac{1}{3} & \frac{1}{3}  \end{bmatrix}
\end{equation}
form a dual frame  $F_d$ for $F$ with the same property: $\mu(Gr(F, F_d))=\frac{1}{3}$. 

\end{example}


\vspace{2.3mm}

The canonical dual $\tilde{F}$ of a given   frame $F$ 
is the natural candidate that would minimize 
$\mu(Gr(F, H))$ with $H=\tilde{F}$, 
but this minimization is not necessarily unique (Example~\ref{exNonSymmetricGram}). 
This uniqueness of minimization is  true for Grassmannian equal norm Parseval frames \cite{BodmannHaas},   however there are non-Parseval frames which  satisfy the same property:

\vspace{2.3mm}




\color{black}

\begin{example}\label{GrassmannEx} The frame $F$ in Example~\ref{niceEx} is a frame for which $\mu(Gr(F, F_d))$ is minimised exclusively at its canonical dual. 
Any dual frame $F_d$ of $F$  is of type 
\begin{equation}
    F_d= \begin{bmatrix}
    x & x-1 & 1-x\\
    y-1 & y & 1-y
    \end{bmatrix}, \,   \text{while} \, \,  \tilde{F}=\frac{1}{3} \begin{bmatrix}
    2 & -1 & 1\\
    -1 & 2 & 1
    \end{bmatrix}
\end{equation}
is its canonical dual ($x=y=2/3$ in $F_d$). The cross Gramians are 
\begin{equation}
    Gr(F, F_d) = \begin{bmatrix}
    x & y-1 & x+y -1 \\
    x-1 & y & x+y -1 \\
    1-x & 1-y & 2-x-y
    \end{bmatrix}, \, \, Gr(F, \tilde{F})= \frac{1}{3} \begin{bmatrix}
    2 & -1 & 1 \\
    -1 & 2 & 1 \\
    1 & 1 & 2
    \end{bmatrix},
\end{equation} so $\mu(Gr(F, \tilde{F})) = \frac{1}{3}$. If  for some values of $x, y$ we have $\mu(Gr(F,  {F}_d)) \leq \frac{1}{3}$, then 
\begin{center} $\vert 1-x \vert \leq \frac{1}{3}$, $\vert 1-y\vert  \leq \frac{1}{3}$ and $\vert x+y -1 \vert \leq \frac{1}{3}$;\end{center}  the only values for $x$ and $y$ which satisfy all three inequalities are $x=y=\frac{2}{3}$, which determine $\tilde{F}$. 
\end{example}
 
  Examples~\ref{exNonSymmetricGram} and \ref{GrassmannEx} indicate  a   lower bound for $\mu$, which would be a generalization of the Welch bound \cite{Welch}; as seen in \cite{ChrisDatta}:
 \begin{lemma}\label{sth}
 Let $F$ be a frame for $\mathbb{F}^n$, and let $G$ be one of its dual frames such that $\langle f_i, g_i \rangle =  \frac{n}{k}$. Then 
\begin{equation}\label{mulowerboundspec} \mu(Gr(F, G)) \geq \sqrt{\frac{nk-n^2}{k^2(k-1)}}, \end{equation} with equality whenever $G$ is the canonical dual of $F$.
 \end{lemma}
  Lemma~\ref{sth} follows directly from  Proposition~\ref{mainresult} and  Proposition~\ref{prop29}.  We  further focus our attention on a special type of a frame pair:
\color{black}
\begin{definition}
 
  A frame $F$  for $\mathbb{F}^n$ forms a  {\it Grassmannian pair} with its dual frame  $\tilde{F}$ if 
\begin{equation}\label{minmu}
    \mu(Gr(F,\tilde{F}))=\min \{\mu(Gr(F,H)) \, | \, H \, \text{ is a dual frame of } \, F\}.
\end{equation}

\end{definition}
In other words, if   $F$ forms a Grassmannian pair with its dual frame $\tilde{F}$, 
then for any dual frame $H$ of $F$ such that $(F, H)$ is not a Grassmannian pair, there exists some $\epsilon > 0$ %
such that 
\begin{equation}\label{tinyEQ}
   \mu(G(F,H))^2 =   \mu(G(F,\tilde{F}))^2 + \epsilon.
\end{equation}
Some frames form an   {\it exclusive} Grassmannian pair with their canonical dual (Example~\ref{GrassmannEx}),  while other frames (Example~\ref{exNonSymmetricGram})  have more than one  dual frame which satisfy \eqref{minmu}.

\vspace{2.3mm}

\subsection{Exponential Potential}

The {\it exponential potential } of $G=Gr(F, H)$ is mapping each pairing $(f_i, h_j)$ into 
\[ E^{\eta}_{i,j} (G) := e^{\eta \vert \langle f_i, h_j \rangle \vert ^2}   \; \text{for any } \; \eta>0.\color{black}\] 
 In particular, the {\it off-diagonal sum potential} of $G$ is
\[\Phi^{\eta}_{od} (G) := \sum_{i=1}^k \sum_{j=1}^k (1 - \delta_{i,j}) E^{\eta}_{i,j} (G),\]   which will  identify a Grassmanian pair as its minimizer (Corollary \ref{grasmannianC}).\color{black}

\begin{proposition}\label{limitprop} Given a frame $F$ for $\mathbb{F}^n$ and a dual frame $H$ of $F$,   it holds
\begin{equation}\label{wantedEq} 
\mu(G)^2 = \lim_{\eta \rightarrow \infty }\frac{1}{\eta} \ln \Phi^{\eta}_{od} (G).\end{equation}
\end{proposition}

\begin{proof}
Let $F$ be a frame for $\mathbb{F}^n$ and let $H$ be a dual frame of $F$. Observe that for the cross Gramian $G=Gr(F,H)$ we have 
 

\begin{align}\label{ineqMuG}
 e^{\eta \mu(G)^2} &= e^{\eta \max_{i\neq j} \vert\langle f_i,h_j\rangle \vert^2} \leq \sum_{i=1}^k \sum_{j=1}^k (1 - \delta_{i,j}) e^{\eta \vert \langle f_i, h_j \rangle \vert^2} \nonumber \\
&=\Phi^{\eta}_{od} (G)  
\leq \sum_{i=1}^k \sum_{j=1}^k (1 - \delta_{i,j}) e^{\eta \mu(G)^2} = k(k-1) e^{\eta \mu(G)^2}.\end{align}

By \eqref{ineqMuG} we have
\[ \ln \left(e^{\eta \mu(G)^2} \right) \leq \ln \left( \Phi^{\eta}_{od} (G)\right) \leq \ln \left(k(k-1)e^{\eta \mu(G)^2} \right),\]
that is 
\[ \eta \mu(G)^2 \leq \ln \left( \Phi^{\eta}_{od} (G)\right) \leq \ln \left(k(k-1) \right) +\ln \left(e^{\eta \mu(G)^2}  \right),\]
so 
\begin{equation}\label{usefulIneq} 
\mu(G)^2 \leq \frac{1}{\eta}\ln \left( \Phi^{\eta}_{od} (G)\right) 
\leq \frac{\ln \left(k(k-1) \right)}{\eta} + \mu(G)^2.\end{equation} 
If we allow $\eta \rightarrow \infty$ in \eqref{usefulIneq}, we get \eqref{wantedEq}. 
\end{proof}
Now, let $F$ form a Grassmannian pair with a dual frame $\tilde{F}$, and let $H$ be any other dual frame of $F$ such that $(F,H)$ is not a Grassmannian pair. We use the notations $G=G(F,\tilde{F})$ and $G'=G(F,H)$. By combining \eqref{tinyEQ} and \eqref{usefulIneq} with the assumption that $\eta>\ln(k(k-1))/\epsilon$,  
we obtain
\begin{equation}\label{usedIneq}
     \mu(G)^2 \leq \frac{1}{\eta}\ln \left( \Phi^{\eta}_{od} (G)\right) <   \mu(G)^2 +\epsilon = \mu(G')^2.\end{equation}
By \eqref{usedIneq} and \eqref{ineqMuG} we see
   $e^{\eta \mu(G)^2} \leq \Phi^{\eta}_{od} (G) < e^{\eta\mu(G')^2} \leq  \Phi^{\eta}_{od} (G')$: \color{black}
 
\begin{corollary}\label{grasmannianC}
If $F$ is a frame for $\mathbb{F}^n$ which  forms a Grassmannian pair with a dual frame $\tilde{F}$, then for $\eta >0$ large enough,
$$\Phi^{\eta}_{od} (G) < \Phi^{\eta}_{od} (G'),$$
where $G=Gr(F,\tilde{F})$ is the cross Gramian with the Grassmannian dual $\tilde{F}$ of $F$, and $G'$ is a cross Gramian with any  dual frame $H$ of $F$ such that $(F,H)$ is not a Grassmannian pair.
\end{corollary}

\vspace{2.3mm}

\begin{definition} Let $F$ be a frame for $\mathbb{F}^n$ with $|F|=k \geq n$ and let $H$ be a dual frame for $F$. Then we define the \textit{sum potential} of $G = Gr(F,H)$ as  $$\Phi^{\eta}_{sum}(G) = \Phi^{\eta}_{od}(G) + \sum_{i=1}^k e^{-\eta (n^2/k^2-C^2_{k,n})}E_{i,i}^\eta (G),$$ where $C_{k,n} = \sqrt{\frac{nk-n^2}{k^2(k-1)}}.$ \end{definition}

Note that $C_{k,n}$ is the magnitude of each off-diagonal entry of $Gr(F,H)$ when equality holds in Theorem \ref{whatsL}.

\begin{proposition} Let $F$ be a frame for $\mathbb{F}^n$, $\vert F \vert = k \geq n$, let $H$ be a dual frame for $F$, and let $G = Gr(F,H)$ be the cross-Grammian of $F$ with $H$. Then 
\begin{equation}\label{expIneq} \Phi^{\eta}_{sum}(G) \geq k^2e^{\eta (\frac{n}{k^2}-\frac{n^2}{k^3}+\frac{n(k-n)}{k^3(k-1)})},
\end{equation} with equality if and only if all of the following hold: \begin{itemize} 
\item[i.] $H$ is the canonical dual for $F$, 

\item[ii.] every off-diagonal entry of $G$ has the same magnitude, and 

\item[iii.] every diagonal entry of $G$ has the same value.

\end{itemize}\end{proposition}

\begin{proof}
Applying Jensen's inequality, we have
$$\Phi^{\eta}_{sum}(G) = \sum_{i,j=1}^k e^{\eta |G_{i,j}|^2 - \delta_{i,j} \eta (\frac{n^2}{k^2}-C_{k,n}^2)} \geq k^2 e^{\frac{\eta}{k^2}(\sum_{i,j=1}^k (|G_{i,j}|^2 - \delta_{i,j} (\frac{n^2}{k^2}-C_{k,n}^2)))},$$
with equality if and only if there exists some $C>0$ such that 
$$|G_{i,j}|^2 - \delta_{i,j}(\frac{n^2}{k^2}-C_{k,n}^2) = C$$
for all $1 \leq i,j \leq k$.
Then, since $\sum_{i,j}^k |G_{i,j}|^2$ is equal to the cross-potential of $H$ with respect to $F$, we have by Proposition \ref{mainresult}:
$$k^2 e^{\frac{\eta}{k^2}(\sum_{i,j=1}^k (|G_{i,j}|^2 - \delta_{i,j} (\frac{n^2}{k^2}-C_{k,n}^2)))} \geq k^2 e^{\eta (\frac{n}{k^2}-\frac{n^2}{k^3}+\frac{n(k-n)}{k^3(k-1)})},$$
with equality if and only if $H$ is the canonical dual frame for $F$.
\end{proof}

\begin{note}\label{noteonoffdiagdecay}

 In the case of equality in \eqref{expIneq}, the magnitude of every off-diagonal entry of $G= Gr(F, H)$ will be $C_{k,n}$, and the value of every diagonal entry will be $n/k$.

\end{note}
 


Given a frame $F=\{f_i\}_{i=1}^k$  for $\mathbb{F}^n$, with a dual  frame $H=\{h_i\}_{i=1}^k$, we define 
\[A^{\alpha}_i := \sum_{j=1}^k e^{\alpha \vert \langle f_j, h_i \rangle \vert^2}.\] 
If $A^{\alpha}_i=A^{\alpha}_l$ for all  $i$, $l \in \{1,\hdots, k\}$, \color{black}  then we say that 
 the frame pair $(F, H)$ is {\it $\alpha$ co-equipartitioned.}

The frame $F$ and its canonical dual $\tilde{F}$ from Example~\ref{niceEx} \color{red} form  \color{black}  an $\alpha$ co-equipartitioned frame pair
 for all $\alpha>0$ as $A^{\alpha}_i= e^{\frac{4}{9}\alpha} + 2e^{\frac{\alpha}{9}}$ for $i \in I=\{1,2,3\}$. These two frames have another interesting property: 
For every $l,l'$ in the index set $I$, there exists a permutation $\pi$ of $I$ such that 
$$ \vert \langle f_j, \tilde{f}_l\rangle \vert = \vert \langle f_{\pi(j)},\tilde{f}_{l'}\rangle \vert \;\; \text{ for all $j$} \in I.$$ 
If a frame $F$ and its dual $H$ satisfy such a permutation property, \color{red} then  \color{black}  we call the frame pair {\it  co-equidistributed}.  Co-equidistributivity and $\alpha$ co-equipartitioning are  closely related:

\begin{corollary}\label{coequies}
If a frame pair $F$ and $H$ \color{red}  is \color{black} co-equidistributed, \color{red} it  is \color{black}   $\alpha$ co-equipartitioned for all $\alpha>0$. 
\end{corollary}
 \begin{proof}
  If   a frame  $F$ and its dual $H$ are co-equidistributed, then the magnitudes of every column in the cross Gramian $G=Gr(F,H)$  are the same as those
of any other column of $G$, up to permutation. By the definition of $\alpha-$ co-equipartitioning, it
follows that $G$ is $\alpha-$equipartitioned for all $\alpha >0$.
 \end{proof} 
 \subsection{Conjectures}
 We conjecture a result more general than Lemma~\ref{sth} holds:\color{black} 

\begin{conjecture} Let $F$ be a frame for $\mathbb{F}^n$, and let $G$ be one of its dual frames. Then 
\begin{equation}\label{mulowerbound} \mu(Gr(F, G)) \geq \sqrt{\frac{nk-n^2}{k^2(k-1)}}. \end{equation}
\end{conjecture}  
  
We present  a partial proof of the conjecture above.  Note that the case when  $  \sum_{i=1}^k \vert \langle f_i, g_i \rangle \vert^2  = n^2/k$ is covered by Lemma~\ref{sth}.  Now, suppose  $  \sum_{i=1}^k \vert \langle f_i, g_i \rangle \vert^2  > n^2/k$. Given that 
$$ n \leq P_F(G)  =\sum_{i=1}^k \vert \langle f_i, g_i \rangle \vert^2 + \sum_{i=1}^k \sum_{j\neq i}  \vert \langle f_i, g_j \rangle \vert^2,$$ we have two lower bounds for $P_F(G)$:
\[ P_F(G) > \frac{n^2}{k} + \sum_{i=1}^k \sum_{j\neq i}  \vert \langle f_i, g_j \rangle \vert^2  \; \text{ and } \; P_F(G) \geq n.\]
It is either that 
\[ \text{(a)} \; \;  
n > \frac{n^2}{k} + \sum_{j\neq i}  \vert \langle f_i, g_j \rangle \vert^2, \; \text{or}  \;\;\;\; \text{(b) }\; \; 
\frac{n^2}{k} + \sum_{j\neq i}  \vert \langle f_i, g_j \rangle \vert^2 \geq n. \;\;\; \]

Note that in the examples we have added to this paper,   it never happens that  (a) is true. 

Now, if (b) is true then we have $ \displaystyle   \sum_{j\neq i}  \vert \langle f_i, g_j \rangle \vert^2 \geq n - \frac{n^2}{k}.$ 
 We take the maximum in each term in the last inequality and get
\[ k(k-1) \mu^2(Gr(F,G))= \sum_{i=1}^k \sum_{j\neq i} \mu^2(Gr(F,G)) = \sum_{i=1}^k \sum_{j\neq i} \max_{i\neq j} \vert \langle f_i, g_j \rangle \vert^2 \geq n - \frac{n^2}{k},  \] thus 
\[\mu^2(Gr(F,G)) \geq \frac{1}{k(k-1)} \left(  n - \frac{n^2}{k}\right), \] that is,  $\mu^2(Gr(F,G)) \geq \frac{n-n^2/k}{k(k-1)}$ and inequality \eqref{mulowerbound}  follows. 

\vspace{2.3mm}

\color{black}
Note that the frame in Example~\ref{GrassmannEx} forms an {\it exclusive} Grassmannian pair with its canonical dual, that is, its canonical dual is the only dual frame that satisfies \eqref{minmu}, while the frame in Example~\ref{exNonSymmetricGram}  has at least two duals which satisfy \eqref{minmu}. We conclude that only frames with special structure form exclusive Grassmannian pairs and we offer a conjecture:

\begin{conjecture} If a  frame $F$ for $\mathbb{F}^n$ forms an exclusive  Grassmannian pair with one of its duals, then that dual must be the canonical dual frame of $F$.
\end{conjecture}

Example~\ref{niceEx} and Example~\ref{GrassmannEx} indicate  the following to be true:

\begin{conjecture}
 A frame $F$ forms  an exclusive Grassmannian pair with a dual frame $\tilde{F}$ if and only if the frame pair $(F,\tilde{F})$  is co-equidistributed. When this is true, the dual frame $\tilde{F}$ is the canonical dual of $F$.

\end{conjecture} 


\section{Cross  fusion frame potential}\label{CFFPdef}
We propose a definition for the cross fusion frame potential and 
 compute its value for \color{black} special classes of fusion frames.
\begin{definition} 
Let $P = \{P_i\}_{i=1}^k$ and $Q = \{Q_{i}\}_{i=1}^k$ be fusion frames for $\mathbb{F}^n$, where $P_i$ and $Q_i$ are orthogonal projections. Then the cross fusion frame potential   of $Q$ with respect to $P$ is 
$$\phi(P,Q) := \sum_{i=1}^k\sum_{j=1}^kTr(P_{i}Q_{j}).$$
\end{definition}

The trace of a matrix product is independent of the order of the matrix product, thus  $\phi(P,Q) = \phi(Q,  P)$. 
Just as the fusion frame potential equals the trace of the square of the frame operator, here we have $$\phi(P,Q) = Tr(S_PS_Q),$$ where $S_P$ is the fusion frame operator for $P$ and $S_Q$ is the fusion frame operator for $Q$: 
\begin{align*}  \phi(P,Q) &= \sum_{i=1}^k\sum_{j=1}^kTr(P_iQ_j) = Tr\left(\sum_{i=1}^k\sum_{j=1}^kP_iQ_j\right)  = Tr\left(\sum_{i=1}^kP_i\left(\sum_{j=1}^kQ_j\right)\right)\\ & = Tr\left(\sum_{i=1}P_iS_Q\right) = Tr\left(\left(\sum_{i=1}^kP_i\right)S_Q\right) = Tr(S_PS_Q).\end{align*}

\begin{definition} \label{semi_orth_def}
Let $W_1, W_2$ be subspaces such that 
\begin{itemize}
    \item[i.] $W_1 \cap W_2 \neq \emptyset$

    \item[ii.] there exists a non-trivial subspace $V_1 \subset W_1$ such that $V_1$ is orthogonal to $W_2$ and $$dim(W_1 \cap W_2) + dim(V_1) = dim(W_1)$$
     
    \item[iii.] there exists a non-trivial subspace $V_2 \subset W_2$ such that $V_2$ is orthogonal to $W_1$ and $$dim(W_1 \cap W_2) + dim(V_2) = dim(W_2).$$
\end{itemize}
Then we say that $W_1$ and $W_2$ are \textit{semi-orthogonal subspaces}.
 
\end{definition}
For instance, the coordinate planes $xy$ and $yz$ are semi-orthogonal in $\mathbb{R}^3$.
Observe that for two semi-orthogonal subspaces $W_1$ and $W_2$ we have: \[P_{W_1} P_{W_2} W_1 =P_{W_2} P_{W_1} W_2 = W_1 \cap W_2.\]

When a fusion frame $P$ has special geometry, then $P$ is its  canonical dual, and we can specify more closely the value of the cross potential:

\begin{theorem} \label{fusion_result}
Let $P = \{W_i\}_{i=1}^k$ be a fusion frame for an $n$-dimensional Hilbert space $\mathbb{F}^n$ such that for every $1 \leq i,j \leq k$ with $i \neq j$, one of the following conditions holds: 

\begin{itemize} 
\item[i.] $W_i = W_j$ 

\item[ii.] $W_i$ and $W_j$ are orthogonal to one another, or

\item[iii.] $W_i$ and $W_j$ are semi-orthogonal to one another.

\end{itemize}
Let $Q$ be the canonical dual of $P$. Then $Q = P$ and 
$$\phi(P,Q)  = FFP_P = \sum_{i=1}^k\sum_{j=1}^kdim(W_i \cap W_j).$$
 
\end{theorem}

\begin{proof}
Let $P$ be as stated and let $Q$ be its canonical dual. Consider some $W_i$ in $P$. Let $F_i$ be a frame for $W_i$, and let $f \in F_i$. We consider what happens when we apply each $P_j$ to $f$ ($P_j$ is the orthogonal projection onto $W_j$ in $P$). 

If $W_i = W_j$, then $P_jf = f$. If $W_i$ and $W_j$ are orthogonal to one another, then $P_jf$ is the zero vector. Finally, if $W_i$ and $W_j$ are semi-orthogonal to one another, then $P_jf$ is equal to $f$ if $f \in W_j$ or the zero vector if $f \not\in W_j$. Thus, if $S = \{P_j\}_{j=1}^k$ is the frame operator for $P$, then $Sf = mf$, where $m > 0$ is an integer. Then $S^{-1}f = \frac{1}{m}f$. Then $S^{-1}W_i = W_i$, and so $Q = P$.

Now let $P_i$ be the orthogonal projection onto a subspace $W_i$ in $P$, and consider the trace of each product $P_iP_j$, when $1 \leq j \leq k$. If $W_i = W_j$, then $P_iP_j = P_i$, and $Tr(P_i) = dim(W_i)$. If $W_i$ and $W_j$ are orthogonal to one another, then $P_iP_j$ contains only zeros, and so $Tr(P_iP_j) = 0$. Finally, if $W_i$ and $W_j$ are semi-orthogonal to one another, then $P_iP_j$ is the orthogonal projection onto $W_i \cap W_j$, and so $Tr(P_iP_j) = dim(W_i \cap W_j)$. The result follows.
\end{proof}

A special class of the type of fusion frame described above are what are called \textit{orthonormal fusion bases} for $\mathbb{F}^n$, where $\mathbb{F}^n$ is an orthogonal sum of the subspaces $W_i$ in the orthonormal fusion basis. 

\begin{corollary} \label{ofb_corollary}
Let $P = \{W_i\}_{i=1}^k$ be an orthonormal fusion basis for $\mathbb{F}^n$, and let $Q$ be the canonical dual for $P$. Then 
$$ \phi(P,Q)  = n.$$
\end{corollary}
\begin{proof}
By Theorem \ref{fusion_result}, we know that the canonical dual $Q$ of $P$ is itself. Since $P$ is an orthogonal fusion basis, whenever $i \neq j$, we have $dim(W_i \cap W_j) = 0$. Thus, $\phi(P,Q) = \sum_{i=1}^kdim(W_i)$. Since $\mathbb{F}^n$ is an orthogonal sum of the subspaces in $P$, the sum of the dimensions of the subspaces is equal to the dimension of the space, $n$.
\end{proof}

We give three examples of computing the cross potential of the canonical dual fusion frame. In the first example, the frame is an orthonormal fusion basis with itself as its canonical dual. In the second, the frame is not an orthonormal fusion basis, but it is a frame of the type described in Theorem \ref{fusion_result}, and so its canonical dual is also itself. Finally, the frame in the third example is not of the type described in Theorem \ref{fusion_result}, and so its canonical dual is not itself, and the cross potential value does not match the value in Theorem \ref{fusion_result}.
\begin{example}
In $\mathbb{R}^3$, let $W_1$ be the $xy$-plane, and let $W_2$ be the $z$-axis. Then $P = \{W_1,W_2\}$ is an orthonormal fusion basis.

The corresponding orthogonal projections are $$P_1 = \begin{bmatrix}
1 & 0 & 0 \\
0 & 1 & 0 \\
0 & 0 & 0
\end{bmatrix}, P_2 = \begin{bmatrix}
0 & 0 & 0 \\
0 & 0 & 0 \\
0 & 0 & 1
\end{bmatrix}.$$
Then the fusion frame operator $S = P_1 + P_2$ is the identity operator $I$, and so $S^{-1}=I$. 
Thus  $S^{-1}f_{i,j}=f_{i,j}$, where $f_{1,i}$ is a frame element of  any frame $F_1$ for $W_1$,  and $f_{2,i}$ is a frame element of  any frame $F_2$ for $W_2$. 
Therefore, our canonical dual fusion frame $Q$ for $P$ is the same as $P$, and so 
$$\phi(P,Q)  = FFP(P) = 3.$$ 
\end{example}

In the following example, $P$ is not an orthonormal fusion basis, but it is of the structure stated in Theorem  \ref{fusion_result}, so $P$ is still its own canonical dual, and the cross potential value will be as stated in Theorem \ref{fusion_result}.
\begin{example}
In $\mathbb{R}^3$, let $W_1$ be the $xy$-plane, and let $W_2$ be the plane given by the equation $x + y = 0$. Then $P = \{W_1,W_2\}$ is a fusion frame for $\mathbb{R}^3$. The corresponding orthogonal projections are $$P_1 = \begin{bmatrix}
1 & 0 & 0 \\
0 & 1 & 0 \\
0 & 0 & 0 
\end{bmatrix}, 
P_2 = \begin{bmatrix}
1/2 & -1/2 & 0 \\
-1/2 & 1/2 & 0 \\
0 & 0 & 1
\end{bmatrix}.$$
The fusion frame operator for $P$ is 
$$S = P_1 + P_2 = \begin{bmatrix}
3/2 & -1/2 & 0 \\
-1/2 & 3/2 & 0 \\
0 & 0 & 1
\end{bmatrix}, \; \text{and so} \; S^{-1} = \begin{bmatrix}
3/4 & 1/4 & 0 \\
1/4 & 3/4 & 0 \\
0 & 0 & 1 
\end{bmatrix}.$$
We now create bases for each $W_i$: 
$$F_1 = \Bigg\{ \begin{bmatrix}
1 \\
1 \\
0
\end{bmatrix}, \begin{bmatrix}
1 \\
-1 \\
0
\end{bmatrix} \Bigg\}, F_2 = \Bigg\{ \begin{bmatrix}
1 \\
-1 \\
0
\end{bmatrix}, \begin{bmatrix}
0 \\
0 \\
1
\end{bmatrix}\Bigg\}.$$
We apply $S^{-1}$ to each vector in $F_1$ and $F_2$ to get frames $G_1$ and $G_2$ for our canonical dual subspaces: 
$$G_1 = \Bigg\{ \begin{bmatrix}
1 \\
1 \\
0
\end{bmatrix}, \begin{bmatrix}
1/2 \\
-1/2\\
0
\end{bmatrix} \Bigg\}, 
 G_2 = \Bigg\{ \begin{bmatrix}
1/2 \\
-1/2 \\
0
\end{bmatrix}, \begin{bmatrix}
0 \\
0 \\
1
\end{bmatrix} \Bigg\}.$$
$G_1$ spans the $xy$-plane, and $G_2$ spans the plane $x + y = 0$, so our canonical dual fusion frame $Q$ for $P$ is the same as $P$. Thus $\phi(P,Q)  = FFP_P(P) = 6$. 

Note that while the cross potential and regular potential share the same value in this example, that value is not the minimal value for $FFP(P)$, which is $\frac{1}{n}\left(\sum_{i=1}^kL_i\right)^2 = \frac{1}{3}(2+2)^2 = \frac{16}{3}$. This tells us that $P$ is non-tight. 
\end{example}

Finally, we give an example of a fusion frame where Theorem \ref{fusion_result} does not apply. Here the canonical dual of our fusion frame $P$ is \textit{not} equal to $P$.


\begin{example}
In $\mathbb{R}^3$, let $W_1$ be the $xy$-plane, and let $W_2$ be the plane $y = z$. Then $P = \{W_1,W_2\}$ is a fusion frame, and the corresponding orthogonal projections are:
$$P_1 = \begin{bmatrix}
1 & 0 & 0 \\
0 & 1 & 0 \\
0 & 0 & 0
\end{bmatrix}, 
P_2 = \begin{bmatrix}
1 & 0 & 0 \\
0 & 1/2 & 1/2 \\
0 & 1/2 & 1/2
\end{bmatrix}.$$ 
Then our fusion frame operator and its inverse are:
$$S = P_1 + P_2 = \begin{bmatrix}
2 & 0 & 0 \\
0 & 3/2 & 1/2 \\
0 & 1/2 & 1/2
\end{bmatrix}, \;\; 
S^{-1} = \begin{bmatrix}
1/2 & 0 & 0 \\
0 & 1 & -1 \\
0 & -1 & 3
\end{bmatrix}.$$
We now choose bases for our subspaces:
$$F_1 = \Bigg\{ \begin{bmatrix}
1 \\
0 \\
0
\end{bmatrix}, \begin{bmatrix}
0 \\
1 \\
0
\end{bmatrix} \Bigg\}, 
F_2 = \Bigg\{ \begin{bmatrix}
0 \\
1 \\
1 
\end{bmatrix}, \begin{bmatrix}
1 \\
0 \\
0
\end{bmatrix} \Bigg\}.$$
Then we apply $S^{-1}$ to each of the vectors in our frames to obtain frames $G_1$ and $G_2$ for the subspaces of our canonical dual:
$$G_1 = \Bigg\{ \begin{bmatrix}
1/2 \\
0 \\
0 
\end{bmatrix}, \begin{bmatrix}
0 \\
1 \\
-1
\end{bmatrix} \Bigg\},
G_2 = \Bigg\{ \begin{bmatrix}
0 \\
0 \\
2
\end{bmatrix}, \begin{bmatrix}
1/2 \\
0 \\
0
\end{bmatrix} \Bigg\}.$$
$G_1$ spans the plane $y+z = 0$, and $G_2$ spans the plane $y = 0$, so the orthogonal projections of our canonical dual frame $Q$ are:
$$Q_1 = \begin{bmatrix}
1 & 0 & 0 \\
0 & 1/2 & -1/2 \\
0 & -1/2 & 1/2
\end{bmatrix}, 
Q_2 = \begin{bmatrix}
1 & 0 & 0 \\
0 & 0 & 0 \\
0 & 0 & 1
\end{bmatrix}.$$
Then $\phi(P,Q)  = 5$. Note that $FFP(P) = 7$, and the minimum value of $FFP(P)$ for a fusion frame for $\mathbb{R}^3$ with 2 subspaces of dimension $2$ is $16/3$. Thus the cross potential of $Q$ with respect to $P$ is not equal to the frame potential of $P$ by itself, and the cross potential is less than the minimum frame potential value.




\color{black}
\end{example}
 
\begin{proposition}
Let $P = \{W_i\}_{i=1}^k$ and $Q = \{V_i\}_{i=1}^k$ be fusion frames for $\mathbb{F}^n$, and let $U:\mathbb{F}^n \rightarrow \mathbb{F}^n$ be a unitary or orthogonal operator. Then \begin{center} $UP = \{UW_i\}_{i=1}^k$ and $UQ = \{UV_i\}_{i=1}^k$ \end{center} are also fusion frames for $\mathbb{F}^n$, and $\phi(P,Q)  = \phi(UP,UQ)$. 
\end{proposition}

\begin{proof}
Let $\{P_i\}_{i=1}^k$ be the associated orthogonal projections for $P$, and let $\{Q_i\}_{i=1}^k$ be the associated orthogonal projections for $Q$. Fix $1 \leq i,j \leq k$. Then $P_i$ can be written as $AA^*$, where $A$ is a matrix whose columns form an orthonormal basis for $W_i$. Then the columns of $UA$ will form an orthonormal basis for $UW_i$, since $U$ preserves inner products. Thus the orthogonal projection in $\mathbb{F}^n$ onto $UW_i$ will be $(UA)(UA)^* = UP_iU^*$. Similarly, the orthogonal projection onto $V_j$ will be $(UB)(UB)^* = UQ_iU^*$, where $B$ is a matrix whose columns form an orthonormal basis for $V_j$.

First we show that $UP$ is a fusion frame. Let $0 < A,B < \infty$ be constants that satisfy the inequality in Definition \ref{fusion_def} for $P$, and let $f \in \mathbb{F}^n$. Then since $U^*f \in \mathbb{F}^n$, 
$$A\lVert U^*f \rVert ^2 \leq \sum_{i=1}^k \lVert P_iU^*f \rVert ^2 \leq B \lVert U^*f \rVert ^2.$$ As $U$ and $U^*$ preserve inner products: 
$A \lVert f \rVert ^2 \leq \sum_{i=1}^k \lVert UP_iU^*f \rVert ^2 \leq B \lVert f \rVert ^2.$ 

Since $UP_iU^*$ is the orthogonal projection onto $UW_i$, we know that $UP$ is a fusion frame. Similarly, we can show that $UQ$ is a fusion frame.

Now consider the product between the orthogonal projections onto $UW_i$ and $UV_j$:
$$(UP_iU^*)(UQ_jU^*) = UP_i(U^*U)Q_jU^* = UP_iQ_jU^*.$$
  \begin{align*}
\text{Then}\;\;  Tr((UP_iU^*)(UQ_jU^*)) &= Tr((UP_i)(Q_jU^*)) \\
 &= Tr((Q_jU^*)(UP_i)) = Tr(Q_jP_i) = Tr(P_iQ_j),    
\end{align*} and so the cross-frame potential values will be equivalent. 
\end{proof}
 
\vspace{-4.3mm}

\end{document}